\documentclass[11pt]{article}
\usepackage{amsmath,amsthm,amssymb}

\def\titlerunning#1{\gdef\titrun{#1}}
\makeatletter
\def\author#1{\gdef\autrun{\def\and{\unskip, }#1}\gdef\@author{#1}}
\def\address#1{{\def\and{\\\hspace*{18pt}}\renewcommand{\thefootnote}{}%
\footnote {#1}}%
\markboth{\autrun}{\titrun}}
\makeatother
\def\email#1{\hspace*{4pt}{\em e-mail}: #1}
\def\MSC#1{{\renewcommand{\thefootnote}{}%
\footnote{\emph{Mathematics Subject Classification (2010):} #1}}}
\def\keywords#1{\par\medskip
\noindent\textbf{Keywords:} #1}

\newtheorem{theorem}{Theorem}[section]
\newtheorem{prop}[theorem]{Proposition}
\newtheorem{cor}[theorem]{Corollary}
\newtheorem{lemma}[theorem]{Lemma}

\newtheorem{defin}[theorem]{Definition}

\theoremstyle{definition}

\newtheorem{remark}[theorem]{Remark}

\numberwithin{equation}{section}

\def\cL{\mathcal L}

\def\cE{\mathcal E}
\def\cF{\mathcal F}

\def\cH{\mathcal H}

\def\cM{\mathcal M}
\def\cK{\mathcal K}
\def\cO{\mathcal O}
\def\cP{\mathcal P}
\def\cX{\mathcal X}

\def\cT{\mathcal T}
\def\cR{\mathcal R}
\def\cS{\mathcal S}
\def\cW{\mathcal W}
\def\cQ{\mathcal Q}

\def\PG{{\rm PG}}

\def\GF{{\rm GF}}

\def\PGL{{\rm PGL}}
\def\PSL{{\rm PSL}}
\def\PGO{{\rm PGO}}

\def\PGU{{\rm PGU}}

\def\PSU{{\rm PSU}}
\def\PSp{{\rm PSp}}

\def\F{\mathbb F}

\begin{document}


\baselineskip=16pt

\titlerunning{}

\title{On line covers of finite projective and polar spaces}

\author{Antonio Cossidente
\and
Francesco Pavese}

\date{}

\maketitle

\address{A. Cossidente: Dipartimento di Matematica, Informatica ed Economia, Universit{\`a} degli Studi della Basilicata, Contrada Macchia Romana, 85100, Potenza, Italy;
\email{antonio.cossidente@unibas.it}
\and
F. Pavese: Dipartimento di Meccanica, Matematica e Management, Politecnico di Bari, Via Orabona 4, 70125 Bari, Italy; 
\email{francesco.pavese@poliba.it}}

\bigskip

\MSC{Primary 51E12; Secondary 51E20 51A50}


\begin{abstract}
An {\em $m$--cover} of lines of a finite projective space $\PG(r,q)$  (of a finite polar space $\cP$) is a set of lines $\cL$ of $\PG(r,q)$ (of $\cP$) such that every point of $\PG(r,q)$ (of $\cP$)  contains $m$ lines of $\cL$, for some $m$. Embed $\PG(r,q)$ in $\PG(r,q^2)$. Let ${\bar \cL}$ denote the set of points of $\PG(r,q^2)$ lying on the extended lines of $\cL$. 

An $m$--cover $\cL$ of $\PG(r,q)$ is an $(r-2)$--dual $m$--cover if there are two possibilities for the number of lines of $\cL$ contained in an $(r-2)$--space of $\PG(r,q)$. Basing on this notion,  we characterize $m$--covers $\cL$ of $\PG(r,q)$ such that ${\bar \cL}$ is a two--character set of $\PG(r,q^2)$. In particular, we show that if $\cL$ is invariant under a Singer cyclic group of $\PG(r,q)$ then it is an $(r-2)$--dual $m$--cover.

Assuming that the lines of $\cL$ are lines of a symplectic polar space $\cW(r,q)$ (of an orthogonal polar space $\cQ(r,q)$ of parabolic type), similarly to the projective case we introduce the notion of an $(r-2)$--dual $m$--cover of symplectic type (of parabolic type).   We prove that an $m$--cover $\cL$ of $\cW(r,q)$ (of $\cQ(r,q)$) has this dual property if and only if $\bar\cL$ is a tight set of an Hermitian variety $\cH(r,q^2)$ or of $\cW(r,q^2)$ (of $\cH(r,q^2)$ or of $\cQ(r,q^2)$). We also provide some interesting examples of $(4n-3)$--dual $m$--covers of symplectic type of $\cW(4n-1,q)$.

\end{abstract}

\section{Introduction}

In this paper we deal with finite projective and polar spaces with special emphasis on finite Hermitian polar spaces.  Polar spaces arise from a finite vector space equipped with a non--degenerate reflexive sesquilinear form. A maximal totally isotropic/totally singular subspace ({\em generator}) of a polar space is a subspace of the underlying vector space of maximal Witt index.

We will use the term $n$--space to denote an $n$--dimensional projective subspace of the ambient projective space. Also in the sequel we will use the following notation $\theta_{n,q} := \genfrac{[}{]}{0pt}{}{n+1}{1}_q=q^n + \ldots + q + 1$ and $\theta_{-1,q} = 0$.

The notion of tight set of a finite generalized quadrangle has been introduced by Payne \cite{SP} and later on was generalized to finite polar spaces in \cite{KD} and in \cite{BKLP}. If $\cT$ is a set of points of a finite polar space, then $\cT$ is said to be {\em tight} if the average number of points of $\cT$ collinear with a given point attains a maximum possible value. The size of a tight set $\cT$ of a finite polar space $\cal P$ in $\PG(r,q)$ is $i\theta_{n,q}$, where a generator of $\cal P$ is an $n$--space, and $\cT$ is said to be $i$--tight. Let $\cT_1$ and $\cT_2$ be $i$--tight and $j$--tight sets of points of $\cP$, respectively. If $\cT_1\subseteq \cT_2$, then $\cT_2\setminus\cT_1$ is a $(j-i)$--tight set. If $\cT_1$ and $\cT_2$ are disjoint, then $\cT_1\cup\cT_2$ is an $(i+j)$--tight set. An $i$--tight set is {\em reducible} if it contains a smaller $i'$--tight set for some integer $i'<i$. A tight set is {\em irreducible} if it is not reducible.

A subset of points $\cO$ of a finite polar space $\cP$ is an $m$--{\em ovoid} if each generator of $\cP$ meets $\cO$ in $m$ points. 

A {\em two--character set} $\cX$ of $\PG(n,q)$  is a subset of points with the property that the intersection number with any hyperplane only takes two values.

An {\em $m$--cover} of lines of a finite projective space $\PG(r,q)$  (of a finite polar space $\cP$) is a set of lines $\cL$ of $\PG(r,q)$ (of $\cP$) such that every point of $\PG(r,q)$ (of $\cP$)  contains $m$ lines of $\cL$, for some $m$. Embed $\PG(r,q)$ in $\PG(r,q^2)$. Let ${\bar \cL}$ denote the set of points of $\PG(r,q^2)$ lying on the extended lines of $\cL$. 

In Section \ref{projective} we introduce the following notion. An $m$--cover $\cL$ of $\PG(r,q)$ is an $(r-2)$--dual $m$--cover if there are two possibilities for the number of lines of $\cL$ contained in an $(r-2)$--space of $\PG(r,q)$. Basing on this notion,  we characterize $m$--covers $\cL$ of $\PG(r,q)$ such that ${\bar \cL}$ is a two--character set of $\PG(r,q^2)$. In particular, we show that if $\cL$ is invariant under a Singer cyclic group of $\PG(r,q)$, then it is an $(r-2)$--dual $m$--cover.

In Section \ref{symplectic} and Section \ref{parabolic} we assume that the lines of $\cL$ are lines of a symplectic polar space $\cW(r,q)$ (of an orthogonal polar space $\cQ(r,q)$ of parabolic type), and similarly to the projective case we introduce the notion of an $(r-2)$--dual $m$--cover of symplectic type (of parabolic type). We prove that an $m$--cover $\cL$ of $\cW(r,q)$ (of $\cQ(r,q)$) has this dual property if and only if $\bar\cL$ is a tight set of an Hermitian variety $\cH(r,q^2)$ or of $\cW(r,q^2)$ (of $\cH(r,q^2)$ or of $\cQ(r,q^2)$). We also provide some interesting examples of $(4n-3)$--dual $m$--covers of symplectic type of $\cW(4n-1,q)$.
At the end of the paper examples of transitive $(q+1)$--tight sets of $\cH(n,q^2)$, $(n,q)\in \{(4,2),(4,3),(6,2)\}$ that are neither parabolic quadrics or symplectic spaces nor the disjoint union of generators, are provided.

\section{$(r-2)$--dual covers of $\PG(r,q)$}\label{projective}

An {\em $m$--cover} of lines of a projective space $\PG(r,q)$ is a set of lines $\cL$ of $\PG(r,q)$ such that every point of $\PG(r,q)$ contains $m$ lines of $\cL$, where $0 < m < \theta_{r-1,q}$. We have the following results.

\begin{lemma}\label{lemma1}
Let $\cL$ be an $m$--cover of $\PG(r,q)$, then
\begin{itemize} 
\item[$i)$] $\cL$ contains $m \theta_{r,q}/(q+1)$ lines, 
\item[$ii)$] every hyperplane of $\PG(r,q)$ contains $m \theta_{r-2,q}/(q+1)$ lines of $\cL$, 
\item[$iii)$] let $\Sigma$ be an $(r-2)$--space of $\PG(r,q)$, if $x$ is the number of lines of $\cL$ contained in $\Sigma$ and $y$ is the number of lines of $\cL$ meeting $\Sigma$ in one point, then $x(q+1)+y = m \theta_{r-2,q}$.  
\end{itemize}
\end{lemma}
\begin{proof}
{\em i)} A standard double counting argument on pairs $(P, \ell)$, where $P$ is a point of $\PG(r,q)$ and $\ell$ is a line of $\cL$ such that $P \in \ell$, gives $|\cL| = m \theta_{r,q}/(q+1)$. Hence if $r$ is even, then $q+1$ divides $m$.

{\em ii)} Let denote with $H_i$, $1 \le i \le \theta_{r,q}$, the hyperplanes of $\PG(r,q)$ and let $x_i$ be the number of lines of $\cL$ contained in $H_i$. Hence the number of lines of $\cL$ meeting $H_i$ in one point equals $ m \theta_{r,q}/(q+1) - x_i$. Double counting the pairs $(P, \ell)$, where $P \in H_i$, $\ell \in \cL$ and $P \in \ell$, we obtain 
$$
m \theta_{r-1,q} = x_i (q+1) + \frac{m \theta_{r,q}}{q+1} - x_i.
$$
It follows that
$$
x_i =  \frac{m \theta_{r-2,q}}{q+1} , 1 \le i \le \theta_{r,q} .
$$

{\em iii)} Let $\Sigma$ be an $(r-2)$--space of $\PG(r,q)$. Let $x$ be the number of lines of $\cL$ contained in $\Sigma$ and let $y$ be the number of lines of $\cL$ meeting $\Sigma$ in one point. A line meeting $\Sigma$ in a point lies in some hyperplane containing $\Sigma$. Since there are $q+1$ hyperplanes of $\PG(r,q)$ through $\Sigma$, it follows that 
$$
y = \left(\frac{m \theta_{r-2,q}}{q+1} - x\right) (q+1).
$$
\end{proof}

An {\em $n$--spread} $\cF$ of $\PG(r,\F)$ is a set of $n$--spaces of $\PG(r,\F)$ such that every point of $\PG(r,\F)$ is contained in exactly one element of $\cF$. In \cite{BF} the authors investigate {\em dual $(r-1)/2$--spreads} $\cal F$ of $\PG(r,\F)$, $r$ odd, namely spreads with the property that any hyperplane of $\PG(r,\F)$ contains exactly one member of $\cal F$. In particular they show that, if $\F$ is finite, then $S$ is always a dual spread (Theorem 1). Generalizing \cite{BF}, we say that an $m$--cover $\cL$ 
of $\PG(r,q)$ is {\em dual} if every hyperplane of $\PG(r,q)$ contains a constant number of lines of $\cL$. From Lemma \ref{lemma1} $ii)$, it follows that any $m$--cover is dual. We introduce the following definition.

\begin{defin}
An $m$--cover of lines $\cL$ of $\PG(r,q)$ is an {\em $(r-2)$--dual $m$--cover} of $\PG(r,q)$, if every $(r-2)$--space of $\PG(r,q)$ contains either $\frac{m \theta_{r-4,q}}{q+1}+q^{r-3}$ lines of $\cL$ or $\frac{m \theta_{r-4,q}}{q+1}$ lines of $\cL$.
\end{defin}

Let $\cL$ be an $m$--cover of $\PG(r,q)$. Embed $\PG(r,q)$ in $\PG(r,q^2)$. For every line $\ell$ of $\PG(r,q)$ we denote by ${\bar \ell}$ the extended line of $\PG(r,q^2)$, i.e., the unique line of $\PG(r,q^2)$ meeting $\ell$ in $q+1$ points. Let ${\bar \cL}$ be the set of points of $\PG(r,q^2)$ lying on the extended lines of $\cL$. Note that $\PG(r,q)$ is contained in ${\bar \cL}$. The following result gives a characterization of the $m$--covers of lines $\cL$ of $\PG(r,q)$ having the property that ${\bar \cL}$ is a two--character set of $\PG(r,q^2)$.

\begin{theorem}\label{main}
Let $\cL$ be an $m$--cover of lines of $\PG(r,q)$ and let ${\bar \cL}$ be the set of points of $\PG(r,q^2)$ lying on the extended lines of $\cL$. Then ${\bar \cL}$ is a two--character set of $\PG(r,q^2)$ if and only if $\cL$ is an $(r-2)$--dual $m$--cover of $\PG(r,q)$. 
\end{theorem}
\begin{proof}
Assume first that $\cL$ is an $(r-2)$--dual $m$--cover of $\PG(r,q)$. Let $\Pi$ be a hyperplane of $\PG(r,q^2)$. From \cite{S}, we have that $\Pi$ meets $\PG(r,q)$ either in a hyperplane or in an $(r-2)$--space. If the former case occurs, from Lemma \ref{lemma1}, $i)$, we have  
$$
|\Pi \cap {\bar \cL}| = (q^2-q) \frac{m \theta_{r-2,q}}{q+1} + \theta_{r-1,q} .
$$  
If the latter case occurs, then let $\Sigma = \Pi \cap \PG(r,q)$. Two possibilities for $\Sigma$ occur. Either $\Sigma$ contains $\frac{m \theta_{r-4,q}}{q+1}+q^{r-3}$ lines of $\cL$ or $\Sigma$ contains $\frac{m \theta_{r-4,q}}{q+1}$ lines of $\cL$. From Lemma \ref{lemma1}, $iii)$, it follows that 
$|\Pi \cap {\bar \cL}|$ equals either
$$
(q^2-q) \frac{m \theta_{r-2,q}}{q+1} + \theta_{r-1,q} 
$$ 
or 
$$
(q^2-q) \frac{m \theta_{r-2,q}}{q+1} + \theta_{r-2,q} .
$$

Viceversa, suppose that every hyperplane of $\PG(r,q^2)$ contains either $\alpha$ or $\beta$ points of ${\bar \cL}$. Set $k = |{\bar \cL}|$. It follows that
\begin{equation}\label{eq2}
k = (q^2-q) |\cL| + \theta_{r,q} = (q^2-q) \frac{m \theta_{r,q}}{q+1} + \theta_{r,q} .
\end{equation}
Note that, if $\Pi$ is a hyperplane of $\PG(r,q^2)$ meeting $\PG(r,q)$ in a $\PG(r-1,q)$, then 
\begin{equation} \label{eq3}
\alpha = |\Pi \cap {\bar \cL}| = (q^2-q) \frac{m \theta_{r-2,q}}{q+1} + \theta_{r-1,q} . 
\end{equation}
From \cite[p. 30--31]{SRG}, we have that 
\begin{equation} \label{eq4}
k^2 \theta_{r-2,q^2} + k (1-\alpha-\beta) \theta_{r-1,q^2} - k \theta_{r-2,q^2} + \alpha \beta \theta_{r,q^2} = 0 . 
\end{equation} 
Taking into account \eqref{eq2}, \eqref{eq3} and substituting $k$ and $\alpha$ in \eqref{eq4} we obtain
\begin{equation} \label{eq5}
\beta = (q^2-q) \frac{m \theta_{r-2,q}}{q+1} + \theta_{r-2,q} .
\end{equation}
On the other hand, if $\Pi$ is a hyperplane of $\PG(r,q^2)$ meeting $\PG(r,q)$ in a $\PG(r-2,q)$, say $\Sigma$, and $x$ is the number of lines of $\cL$ contained in $\Sigma$ and $y$ is the number of lines of $\cL$ meeting $\Sigma$ in one point, then 
\begin{equation} \label{eq6}
|\Pi \cap {\bar \cL}| = \frac{m \theta_{r,q}}{q+1} - x - y - (q^2-q) x + \theta_{r-2,q} = \frac{m \theta_{r,q}}{q+1} - (m-1) \theta_{r-2,q} + x q^2 . 
\end{equation}
Comparing \eqref{eq6} with \eqref{eq3}, we obtain 
$$
x = \frac{m \theta_{r-4,q}}{q+1} + q^{r-3} ,
$$
whereas, comparing \eqref{eq6} with \eqref{eq5}, we get
$$
x = \frac{m \theta_{r-4,q}}{q+1} .
$$
\end{proof}

If $r = 3$, then every $m$--cover of $\PG(3,q)$ is $(r-2)$--dual. However this is no longer true if $r > 3$, see Remark \eqref{ex}. 

\subsection{$m$--covers of lines of Singer type}

In this section we exhibit some examples of $(r-2)$--dual $m$--covers of $\PG(r,q)$. A {\em Singer cyclic group} of $\PG(r,q)$ is a cyclic group $\cal K$ of $\PGL(r+1,q)$ acting regularly on points of $\PG(r,q)$ \cite{Huppert}. In \cite[Lemma 2.1, Lemma 2.2]{Drudge} the author studies the action of $\cal K$ on subspaces of $\PG(r,q)$ and counts the number of orbits. In particular, when $\cal K$ acts on lines of $\PG(r,q)$ and $r$ is odd, apart from one orbit which is a line--spread (or a $1$--cover), any other orbit of $\cal K$ is a $(q+1)$--cover of $\PG(r,q)$. When $r$ is even all $\cal K$--orbits on lines are $(q+1)$--covers. We use the term {\em Singer cover} to denote a $\cK$--orbit of lines of $\PG(r,q)$.

	\begin{lemma}
Any Singer cover of $\PG(r,q)$ is an $(r-2)$--dual cover of $\PG(r,q)$.	
	\end{lemma}
	\begin{proof}
Let $\cL$ be a Singer cover of $\PG(r,q)$. Embed $\PG(r,q)$ in $\PG(r,q^2)$ and let $\bar{\cL}$ be the set of points of $\PG(r,q^2)$ lying on the extended lines of $\cL$. By considering the action of $\cal K$ on points of $\PG(r,q^2)$, it follows that, when $r$ is odd, the projective space $\PG(r,q^2)$ is partitioned into two $\frac{r-1}{2}$--spaces of $\PG(r,q^2)$, $\Sigma_i$, $i=1,2$, and $(q-1) \theta_{\frac{r-1}{2},q^2}$ Baer subgeometries, see \cite[Theorem 4.1]{Mellinger}. Whereas, when $r$ is even, $\PG(r,q^2)$ is partitioned into $\theta_{r,q^2}/\theta_{r,q}$ Baer subgeometries, see \cite[Theorem 4.29]{H}. From \cite{S} a simple counting argument shows that a hyperplane of $\PG(r,q^2)$ (not containing a subspace $\Sigma_i$) meets exactly one Baer subgeometry of the above partition in $\theta_{r,q}$ points.
If $\cL$ is the unique $1$--cover of $\PG(r,q)$, $r$ odd, then $\bar{\cL}$ is the disjoint union of $\Sigma_1$, $\Sigma_2$ and $q-1$ Baer subgeometries, see \cite{C1}. If $\cL$ is a $(q+1)$--cover, then $\bar{\cL}$ is the union of $q^2-q+1$ Baer subgeometries. In both cases, $\bar{\cL}$ is a two--character set. The result follows from Theorem \ref{main}.	
	\end{proof}

	\section{$(r-2)$--dual covers of symplectic type}\label{symplectic}

Let ${\cal W}(r,q)$ be a symplectic polar space of $\PG(r,q)$, $r \ge 3$ odd, and let $\perp$ denote the associated symplectic polarity. In this section we consider the case when the lines of an $m$--cover $\cL$ of $\PG(r,q)$ are totally isotropic with respect to $\perp$, i.e., {\em $m$--covers of $\cW(r,q)$}. We introduce the following definition.

	\begin{defin}
An $m$--cover $\cL$ of ${\cal W}(r,q)$ is said to be {\em $(r-2)$--dual of symplectic type} if 
$$\vert\{r\in \cL:r\subseteq \ell^\perp\}\vert=
	\begin{cases}
\frac{m \theta_{r-4,q}}{q+1}+q^{r-3} & \mbox{ if } \ell\in{\cL}, \\
\frac{m \theta_{r-4,q}}{q+1} & \mbox{ if } \ell\not\in{\cL}. 
	\end{cases}
$$		
	\end{defin}
	
If $\cL$ is the set of all totally isotropic lines of ${\cal W}(r,q)$, from \cite[Section 2]{CK}, the set $\bar{\cL}$ is a non--degenerate Hermitian variety ${\cal H}(r,q^2)$ of $\PG(r,q^2)$ with associated unitary polarity, say $\perp_{h}$, see also \cite[Table 4.5.A]{KL}. In this geometric setting we also consider the embedding of ${\cal W}(r,q)$ in ${\cal W}(r,q^2)$ and, with a slight abuse of notation, we denote by $\perp$ the symplectic polarity associated to ${\cal W}(r,q^2)$. Let $\tau$ denote the semilinear involution of ${\cal W}(r,q^2)$ fixing ${\cal W}(r,q)$ pointwise. It can be seen that $\perp \perp_h = \perp_h \perp = \tau$, i.e., $\perp$ and $\perp_h$ are commuting polarities of the ambient projective space. Let $P$ be a point of ${\cal W}(r,q^2)$. Then the polar hyperplane of $P$ with respect to $\perp$ coincides with the polar hyperplane of $P'=\tau(P)$ with respect to $\perp_h$. In particular, if $P \in \cW(r,q)$, then $P^\perp = P^{\perp_h}$.

A subset of points $\cal X$ of ${\cal H}(r,q^2)$ (or of $\cW(r,q^2)$) is said to be {\em i-tight} if 
$$
\vert P^{\perp_h}\cap {\cal X} \vert=
	\begin{cases}
i \theta_{\frac{r-3}{2},q^2} + q^{r-1} & \mbox{ if  } P\in{\cal X}, \\
i \theta_{\frac{r-3}{2},q^2} & \mbox{ if  } P\not\in{\cal X}.
	\end{cases}
$$			
We have the following result.
	
\begin{theorem}\label{main1}
Let $\cL$ be an $m$--cover of ${\cal W}(r,q)$ and let ${\bar \cL}$ be the set of points of ${\cH}(r,q^2)$ lying on the extended lines of $\cL$. Then ${\bar \cL}$ is an $(m(q^2-q)+q+1)$--tight set of ${\cal H}(r,q^2)$ if and only if $\cL$ is an $(r-2)$--dual cover of symplectic type. 
\end{theorem}
\begin{proof}
Assume that ${\bar\cL}$ is an $(m(q^2-q)+q+1)$--tight set of ${\cal H}(r,q^2)$. Note that if $P\in{\cH}(r,q^2) \setminus {\cal W}(r,q)$ then $P^{\perp_h}\cap{\cW(r,q)}=\ell_P^{\perp}$, where $\ell_P$ is the unique extended line of $\cL$ through $P$. Assume that $P\in{\bar\cL}\setminus{\cal W}(r,q)$. It follows that 
$$
\vert P^{\perp_h}\cap{\bar\cL}\vert=(m(q^2-q)+q+1) \theta_{\frac{r-3}{2},q^2} + q^{r-1} = x (q^2-q) + \theta_{r-2,q} + \vert\cL\vert-x-y, 
$$
where $x$ is the the number of extended lines of $\cL$ contained in $P^{\perp_h}$ (hence contained in $\ell_P^\perp$) and $y$ is the number of extended lines of $\cL$ meeting $P^{\perp_h}$ in a point of $\cW(r,q)$. Then, taking into account Lemma \ref{lemma1}, $iii)$, we have 
$$
xq^2=m\frac{(q^2+1)\theta_{r-2,q}-\theta_{r,q}}{q+1}+q^{r-1}=m q^2 \frac{\theta_{r-4,q}}{q+1}+q^{r-1},
$$
and hence $x= \frac{m \theta_{r-4,q}}{q+1}+q^{r-3}$.
If $P\in{\cH}(r,q^2)\setminus{\bar\cL}$, then 
$$
\vert P^{\perp_h}\cap{\bar\cL}\vert = (m(q^2-q)+q+1)\theta_{\frac{r-3}{2},q^2} = x(q^2-q)+\theta_{{r-2},q} + \vert\cL\vert-x-y,
$$ 
where again $x$ is the the number of extended lines of $\cL$ contained in $P^{\perp_h}$ (hence contained in $\ell_P^{\perp}$) and $y$ is the number of extended lines of $\cL$ meeting $P^{\perp_h}$ in a point of $\cW(r,q)$. In this case $x=\frac{m \theta_{r-4,q}}{q+1}$.

Viceversa, assume that $\cL$ is an $(r-2)$--dual cover of symplectic type. If $P\in{\cal W}(r,q)$, then 
$$
P^{\perp_h}\cap{\bar\cL}=m (q^2-q)\frac{\theta_{r-2,q}}{q+1}+\theta_{r-1,q}=(m(q^2-q)+q+1)\theta_{\frac{r-3}{2},q^2}+q^{r-1}.
$$
On the other hand, if $P\in{\cH}(r,q^2)\setminus{\cW}(r,q)$, then 
$$
P^{\perp_h}\cap{\bar\cL} = x(q^2-q) + \theta_{r-2,q} + \vert\cL\vert - x - y,
$$ 
where $x$ is the the number of lines of $\cL$ contained in $P^{\perp_h}$ (hence contained in $\ell_P^\perp$) and $y$ is the number of lines of $\cL$ meeting $P^\perp$ in a point of $\cW(r,q)$. Hence, ${\bar \cL}$ is an $(m(q^2-q)+q+1)$--tight set of ${\cal H}(r,q^2)$.
\end{proof}

	\begin{cor}
In $\PG(r,q)$, $r$ odd, if $\cL$ is an $(r-2)$--dual cover of symplectic type, then $\cL$ is an $(r-2)$--dual cover. 
	\end{cor}
	\begin{proof}
If $\cL$ is an $(r-2)$--dual cover of symplectic type, than $\bar\cL$ is a tight set of ${\cH}(r,q^2)$, which in turn is a two--character set of $\PG(r,q^2)$, see \cite[Theorem 12]{BKLP}. From Theorem \ref{main}, $\cL$ is an $(r-2)$--dual cover. 
	\end{proof}

	\begin{remark}
The tight sets constructed in Theorem \ref{main1} are all reducible since they always contain the $\theta_{r,q}$ points of $\cW(r,q)$ which constitute a $(q+1)$--tight set of ${\cH}(r,q^2)$.	
	\end{remark}
	
Analogously to Theorem \ref{main1} we have the following result.	
	\begin{theorem}\label{main2}
Let $\cL$ be an $m$--cover of ${\cal W}(r,q)$ and let ${\bar \cL}$ be the set of points of ${\cal W}(r,q^2)$ lying on the extended lines of $\cL$. Then ${\bar \cL}$ is an $(m(q^2-q)+q+1)$--tight set of ${\cal W}(r,q^2)$ if and only if $\cL$ is an $(r-2)$--dual cover of symplectic type.
	\end{theorem}	
	\begin{proof}
Assume that ${\bar\cL}$ is an $(m(q^2-q)+q+1)$--tight set of ${\cal W}(r,q^2)$. By construction, ${\bar\cL}$ is contained in ${\cal H}(r,q^2)$. Also $P^\perp = \tau(P)^{\perp_h}$ and $P \in {\bar \cL}$ if and only if $\tau(P) \in {\bar\cL}$. Hence ${\bar \cL}$ is an $(m(q^2-q)+q+1)$--tight set of ${\cH}(r,q^2)$ and, from Theorem \ref{main1}, $\cL$ is an $(r-2)$--dual cover of symplectic type. 

Viceversa, assume that $\cL$ is an $(r-2)$--dual cover of symplectic type. If $P\in{\cal W}(r,q)$, then 
$$
\vert P^{\perp}\cap{\bar\cL}\vert=\vert P^{\perp_h}\cap{\bar\cL}\vert=m (q^2-q) \frac{\theta_{r-2,q}}{q+1}+\theta_{r-1,q}=(m(q^2-q)+q+1) \theta_{\frac{r-3}{2},q^2} + q^{r-1}. 
$$
If $P\in{\cH}(r,q^2)\setminus{\cal W}(r,q)$ then $P^{\perp}={\tau(P)}^{\perp_h}$, where $\langle P,\tau(P)\rangle$ is the unique extended totally isotropic line of ${\cal W}(r,q)$ through $P$. Hence $\tau(P)\in\bar\cL$ if and only if $P\in\bar\cL$. It follows that $\vert P^\perp\cap{\bar\cL}\vert$ is either 
$$
(m(q^2-q)+q+1) \theta_{\frac{r-3}{2},q^2} + q^{r-1}
$$ 
or 
$$
(m(q^2-q)+q+1) \theta_{\frac{r-3}{2},q^2},
$$
according as $P$ lies or does not lie on $\bar\cL$. Finally, if $P\in{\cal W}(r,q^2)\setminus {\cal H}(r,q^2)$ then $\tau(P)\not\in {\cal W}(r,q)$ and the line $\langle P,\tau(P)\rangle$ is not a line of ${\cal W}(r,q^2)$. Now $P^\perp=\tau(P)^{\perp_h}$ and $\tau(P)^{\perp_h}$ meets ${\cal H}(r,q^2)$ in a Hermitian variety ${\cal H}(r-1,q^2)$ that is a $\theta_{\frac{r-3}{2},q^2}$--ovoid of ${\cal H}(r,q^2)$ \cite[Lemma 7]{BKLP}. From \cite[Corollary 5]{BKLP}, we have that 
$$
\vert P^\perp\cap
\bar\cL\vert =\vert\tau(P)^{\perp_h}\cap {\bar\cL}\vert = (m(q^2-q)+q+1) \theta_{\frac{r-3}{2},q^2} .
$$
Hence ${\bar \cL}$ is a $(m(q^2-q)+q+1)$--tight set of ${\cal W}(r,q^2)$.	
	\end{proof}

	\begin{remark}
If $\cL$ consists of all the lines of $\cW(r,q)$, then $\cL$ is an $(r-2)$--dual $\theta_{r-2,q}$--cover of symplectic type. A more interesting example arises by considering the set of lines $\cL$ of the split Cayley hexagon $H(q)$ embedded in $\cW(5,q)$, $q$ even. Indeed, from \cite{CMP} the set $\bar\cL$ is a $(q^3+1)$--tight set of both ${\cal H}(5,q^2)$ and ${\cal W}(5,q^2)$. Hence $\cL$ is a $3$--dual $(q+1)$--cover of symplectic type of ${\cal W}(5,q)$, $q$ even. A similar result holds true in the odd characteristic case, see Proposition \ref{splitodd}. 
	\end{remark}

Note that Theorem \ref{main1} and Theorem \ref{main2} generalize Theorem 2.1 of \cite{P}. Indeed, if $r = 3$, then every $m$--cover of $\cW(3,q)$ is $(r-2)$--dual of symplectic type. However this is no longer true if $r > 3$, see Remark \eqref{ex}. 

	\begin{remark}\label{ex}
There exist $m$--covers of $\PG(r,q)$ (of $\cW(r,q)$) that are neither $(r-2)$--dual covers nor $(r-2)$--dual covers of symplectic type. Let $\cW(5,3)$ be a symplectic polar space of $\PG(5,3)$. The group $\PSp(6,3)$ fixing $\cW(5,3)$ has two conjugacy classes of subgroups isomorphic to $\PSL(2,13)$, see also \cite{B}. With the aid of {\em MAGMA} \cite{BCP} we checked that one of them has $8$ orbits on lines of ${\cal W}(5,3)$ of sizes $91,91,364,364,546,546,546,1092$.
All of them are $m$--covers of $\PG(5,3)({\cal W}(5,3))$, for some $m$. The orbits of size $91$ are line spreads of $\PG(5,3)$ (of $\cW(5,3)$) that are neither $3$--dual nor $3$--dual of symplectic type. One of the two orbits of size $364$ is a $3$--dual cover of symplectic type and hence $3$--dual. All the remaining orbits are neither $3$--dual covers nor $3$--dual covers of symplectic type. Notice that the union of two orbits of size $546$ and the orbit of size $1092$ is a $3$--dual cover of symplectic type.
	\end{remark}

	\subsection{$(4n-3)$--dual covers of symplectic type of $\PG(4n-1,q)$}
	
In this section we provide some instances of $(4n-3)$--dual covers of symplectic type of $\PG(4n-1,q)$, $n \ge 2$. The projective symplectic group $\PSp(4n,q)$ acts naturally on the projective space $\Sigma := \PG(4n-1,q)$ and fixes a symplectic polar space $\cW(4n-1,q)$ with associated symplectic polarity $\perp$. In \cite{RHD1} R.H. Dye constructed a spread $\cF$ in $\PG(4n-1,q)$ consisting of totally isotropic lines of the polarity $\perp$, whose stabilizer in $\PSp(4n,q)$ contains a group $G$ isomorphic to $\PSp(2n,q^2)$. From \cite[p. 178]{RHD1}, the group $G$ fixes a pencil of linear complexes, say $\cW_i$, $1 \le i \le q+1$, one of them being $\cW(4n-1,q)$.

In \cite[Theorem 1, p. 499]{RHD} the author also proved that, if $n\ge 2$, the group $G$ has three orbits on totally isotropic lines of $\PG(4n-1,q)$, namely the spread $\cF$, the set ${\cal O}_1$  consisting of the lines that are not conjugate to each member of $\cF$ that they intersect non--trivially, and the set ${\cal O}_2$ consisting of the lines that are conjugate to each member of $\cF$ that they intersect non--trivially. An equivalent description of $\cO_i$, $i = 1,2$, is the following: let $\ell$ be a line of $\cW(4n-1,q^2)$, with $\ell \notin \cF$ and let $r_1, \dots, r_{q+1}$ be the lines of $\cF$ meeting $\ell$ in one point. Then $\langle r_1, \dots, r_{q+1} \rangle$ is a solid $\cT$ and there are two possibilities: either $\cT \cap \cW(4n-1,q) = \cW(3,q)$ and $\ell \in \cO_1$ or $\cT \subset \cW(4n-1,q)$ and $\ell \in \cO_2$. 
The group $G$ acts transitively on points of $\PG(4n-1,q)$, see \cite[Theorem 5]{RHD1}. It follows that $\cF$, ${\cal O}_1$ and ${\cal O}_2$ are $m$--covers of $\cW(4n-1,q)$, for some $m$.

Embed $\Sigma$ in $\PG(4n-1,q^2)$ and let $\tau$ denote the semilinear involution of $\PG(4n-1,q^2)$ fixing $\Sigma$ pointwise. Each linear complex $\cW_i$ gives rise to a non--degenerate Hermitian variety ${\cal H}_i$ of $\PG(4n-1,q^2)$ and these Hermitian varieties form a pencil of $\PG(4n-1,q^2)$. Associated with the spread $\cF$ are two disjoint $(2n-1)$--dimensional subspaces of $\PG(4n-1,q^2)$, say $\Sigma_1$ and $\Sigma_2$, that are generators of each Hermitian variety ${\cal H}_i$ and that are disjoint from $\Sigma$. The lines of the spread $\cF$, when extended over $\GF(q^2)$, meet both $\Sigma_1$ and $\Sigma_2$. In particular, the lines of $\cF$ are exactly those lines of $\cW(4n-1,q)$ meeting both $\Sigma_1$ and $\Sigma_2$. Also, $\tau(\Sigma_1) = \Sigma_2$. The group $G$ stabilizes both $\Sigma_1, \Sigma_2$ and acts naturally on $\Sigma_1$ fixing a symplectic polar space $\cW(2n-1,q^2) \subset \Sigma_1$. If $g$ is an $(n-1)$--space of $\Sigma_1$ that is a generator of $\cW(2n-1,q^2)$, then $\langle g, \tau(g) \rangle$ meets $\Sigma$ in a $(2n-1)$--space $\gamma$. Note that $\cF$ induces a line--spread on $\gamma$, and hence $\gamma$ is totally isotropic with respect to each linear complex $\cW_i$, $1 \le i \le q+1$. Varying $g$, we get a set $\cM$ of size $(q^2+1)(q^4+1)\dots(q^{2n}+1)$, consisting of $(2n-1)$--spaces of $\PG(4n-1,q)$ that are totally isotropic with respect to each symplectic space $\cW_i$, $1 \le i \le q+1$. It follows that the lines in common to every symplectic space $\cW_1$, $1 \le i \le q+1$ are in the union $\cF\cup{\cal O}_2$, that are all the lines lying in some member of $\cM$. 

Using \cite[Section 2.2]{CP} it can be seen that, for a point $P\in\Sigma_1\cup\Sigma_2$, the polar hyperplane of the point $P$ is the same with respect to each Hermitian variety ${\cal H}_i$. We denote such a polar hyperplane $P^{\perp_h}$.

Let $P$ be a point of $\Sigma_1$ and let $P^{\perp_h}$ be the polar hyperplane of $P$ with respect to the polarity associated with each ${\cal H}_i$. The intersection of $P^{\perp_h}$ and $\Sigma_2$ is a $(2n-2)$--space of $\Sigma_2$, say $S_P$, and we say that $S_P$ {\em corresponds} to $P$. Note that the subspace $\langle P, S_P\rangle$ is a generator of $\cH_i$, $1 \le i \le q+1$.

Take another point $P'$ of $\Sigma_1$, $P'\ne P$, and suppose that its corresponding subspace $S_{P'}$ coincides with $S_P$. Then, the line joining $P$ with $P'$ and $S_P$ would be orthogonal to each other and hence would generate a $2n$--space contained in $\cH_i$, a contradiction.

Thus, allowing $P$ to vary over the points of $\Sigma_1$, the construction described above produces a family, say ${\cP}_1$, of
$\theta_{2n-1,q^2}$ distinct generators of $\cH_i$, $1 \le i \le q+1$. In a similar way, as $Q$ varies over the points of $\Sigma_2$, one obtains another collection, say ${\cP}_2$, of $\theta_{2n-1,q^2}$ distinct generators of $\cH_i$, $1 \le i \le q+1$. By considering the points lying in some element of $\cP_1$, we get a subset of points of $\cH_i$, $1 \le i \le q+1$, say $\cR$, with
$$
|\cR| = (q^{4n-2}+1) \theta_{2n-1,q^2} 
$$
and containing at least $2\theta_{2n-1,q^2}+2$ generators of $\cH_i$, $1 \le i \le q+1$, including $\Sigma_1$ and $\Sigma_2$.

\begin{theorem}\label{main3}
The set $\cR$ is a $(q^{4n-2}+1)$--tight set of ${\cal H}_i$.
\end{theorem}

\begin{proof}
Let $H$ be a tangent hyperplane to ${\cal H}_i$ at a point $P$. We distinguish several cases.

Assume first that $P\in\Sigma_1$ or $P\in\Sigma_2$. Then in this case either $\Sigma_1\subset H$ or $\Sigma_2 \subset H$. Let $\Sigma_1 \subset H$. A similar argument holds if $\Sigma_2\subset H$. In this case the intersection between $H$ and $\Sigma_2$ is the $(2n-2)$--space $S_P$. Of course, $H$ meets any other member of $\cP_1$ in a $(2n-2)$--space. It follows
that $H$ meets $\cR$ in $h_1$ points where
$$
h_1=\theta_{2n-1,q^2}+\theta_{2n-2,q^2}+(\theta_{2n-1,q^2}-\theta_{2n-2,q^2}-1)+(\theta_{2n-1,q^2}-1)(\theta_{2n-2,q^2}-\theta_{2n-3,q^2}-1)=
$$
$$
= \theta_{2n-1,q^2}+q^{4n-4}(\theta_{2n-1,q^2}-1)=
(q^{4n-2}+1)\theta_{2n-2,q^2}+q^{4n-2}.
$$

Assume that $P\not\in \Sigma_1\cup\Sigma_2$. Let $\ell$ be the unique line containing $P$ and meeting both, $\Sigma_1$ and $\Sigma_2$, in a point. Let $X_i=H \cap\Sigma_i$, $i=1,2$.

If $P\in{\cal R}$, then $\ell$ is a line of ${\cal H}_i$ and $\langle \ell, X_2 \rangle$ is a member of ${\cal P}_1$ contained in $H$. Also $H$ meets $(\theta_{2n-2,q^2}-1)$ members of $\cP_1$ in a $(2n-2)$--space meeting $X_1$ in a point and $X_2$ in a $(2n-3)$--space. The remaining members of $\cP_1$ have empty intersection with $X_1$ and meet $X_2$ and hence each of them shares $\theta_{2n-2,q^2}-\theta_{2n-3,q^2}$ points with $H$. It follows that $H$
meets $\cR$ in $h_1$ points where
$$
h_1=2\theta_{2n-2,q^2}+\theta_{2n-1,q^2}-\theta_{2n-2,q^2}-1+(\theta_{2n-2,q^2}-1)(\theta_{2n-2,q^2} - \theta_{2n-3,q^2}-1) + 
$$
$$
+ (\theta_{2n-1,q^2}-\theta_{2n-2,q^2})(\theta_{2n-2,q^2}-\theta_{2n-3,q^2}) = (q^{4n-2}+1)\theta_{2n-2,q^2}+q^{4n-2}.
$$

If $P\in{\cal H}_i\setminus {\cal R}$ then $\ell$ is a line that is secant to ${\cal H}_i$ and does not exist a member of ${\cal P}_1$ through $P$ contained in $H$. It follows that $H$ meets $\cR$ in $h_2$ points where
$$
h_2=2\theta_{2n-2,q^2}+\theta_{2n-2,q^2}(\theta_{2n-2,q^2}-\theta_{2n-3,q^2}-1)+
(\theta_{2n-1,q^2}-\theta_{2n-2,q^2}-1)(\theta_{2n-2,q^2}-\theta_{2n-3,q^2}) =
$$
$$
=(q^{4n-2}+1)\theta_{2n-2,q^2}.
$$
\end{proof}

Let $\bar{\cal F},\bar{\cO_1},\bar{\cO_2}$ be the sets of points of ${\cH}_i$ lying on the extended lines of ${\cF},{\cO_1},{\cO_2}$, respectively. 

	\begin{prop}
Each of the sets $\bar{\cF},\bar{\cO_1},\bar{\cO_2}$ is an $i$--tight set of $\cH_i$ with parameter $i = q^2+1, q^{4n-1}-q^{4n-2}+q+1, q^{4n-2}-q^2+q+1$, respectively.
	\end{prop}
	\begin{proof}
The set $\bar{\cal F}$ is the disjoint union of $q-1$ Baer subgeometries and of the two generators $\Sigma_1$ and $\Sigma_2$, \cite[Theorem 4.1]{Mellinger}. Since a Baer subgeometry embedded in $\cH_i$ is a $(q+1)$--tight set of $\cH_i$, see \cite[Section 5.2]{BKLP}, and $\Sigma_i$, $i=1,2$, is a $1$--tight set of $\cH_i$, we have that $\bar{\cal F}$ is a $(q^2+1)$--tight set of $\cH_i$. Moreover ${\cR}=\bar{\cF}\cup\bar{\cO_2}$ and $\bar{\cF} \cap \bar{\cO_2}=\Sigma$, hence $\bar{\cO_2}$ is $(q^{4n-2}-q^2+q+1)$--tight set of $\cH_i$. Trivially the whole of ${\cal H}_i$ can be considered a $(q^{4n-1}+1)$--tight set of itself, which gives that $\bar{\cO_1}$ is a $(q^{4n-1}-q^{4n-2}+q+1)$--tight set of $\cH_i$.
	\end{proof}

From Theorem \ref{main1} and the previous proposition we have the following corollary.

	\begin{cor}
Each of the $m$--covers ${\cal F},{\cal O}_1,{\cal O}_2$ of $\cW(4n-1,q)$ is $(4n-3)$--dual of symplectic type.	
	\end{cor}
	
\section{$(r-2)$--dual covers of parabolic type}\label{parabolic}

Let ${\cal Q}(r,q)$ be a parabolic polar space of $\Sigma:=\PG(r,q)$, $q$ odd, $r \ge 4$ even, and let $\perp$ denote the associated orthogonal polarity. Let $G$ be the group of projectivities of $\Sigma$ isomorphic to $\PGO(r+1,q)$ stabilizing the parabolic polar space $\cQ(r,q)$. Embed $\Sigma$ in $\PG(r,q^2)$ and let $\tau$ denote the semilinear involution of $\PG(r,q^2)$ fixing $\Sigma$ pointwise. From \cite[Table 4.5.A]{KL} the group $G$, considered as a group of projectivities of $\PG(r,q^2)$, fixes a non--degenerate Hermitian variety ${\cal H}(r,q^2)$, with associated unitary polarity $\perp_{h}$, and a parabolic quadric $\cQ(r,q^2)$ of $\PG(r,q^2)$, where $\cQ(r,q) \subset \cQ(r,q^2) \cap \cH(r,q^2)$. With a slight abuse of notation, we denote by $\perp$ the orthogonal polarity associated to ${\cal Q}(r,q^2)$. Then $\perp \perp_h = \perp_h \perp = \tau$, i.e., $\perp$ and $\perp_h$ are commuting polarities of the ambient projective space. Let $P$ be a point of ${\cal Q}(r,q^2)$. Then the polar hyperplane of $P$ with respect to $\perp$ coincides with the polar hyperplane of $P'=\tau(P)$ with respect to $\perp_h$. In particular, if $P \in \Sigma$, then $P^\perp = P^{\perp_h}$. 

From \cite[Theorem 8]{BKLP}, the group $G$ has three orbits on points of $\cQ(r,q^2)$, namely $\cQ(r,q^2)$, $\cO$ and $\cO'$. The orbit $\cO$ has size $q(q^{r}-1)(q^{r-2}-1)/(q^2-1)$ and consists of the points on the extended lines of $\cQ(r,q)$ that are not in $\Sigma$, while $\cO'$ has size $q^{r-1}(q^{r}-1)/(q+1)$, and consists of the points lying on the extended lines of $\Sigma$ that are external to $\cQ(r,q)$. 

A subset of points $\cal X$ of ${\cH}(r,q^2)$ (or of $\cQ(r,q^2)$) is said to be {\em i-tight} if 
$$
\vert P^{\perp_h}\cap {\cal X} \vert=
	\begin{cases}
i \theta_{\frac{r-4}{2},q^2} + q^{r-2} & \mbox{ if  } P\in {\cal X}, \\
i \theta_{\frac{r-4}{2},q^2} & \mbox{ if  } P \not\in {\cal X}.
	\end{cases}	
$$

From \cite[Theorem 8]{BKLP} each of the three $G$--orbits $\cQ(r,q)$, $\cO$, $\cO'$ is an $i$--tight set of $\cQ(r,q^2)$ with parameter $i = q+1$, $q^{r-1}-q$, $q^r-q^{r-1}$, respectively. An analogous result will be proved in the first part of this section.

\begin{prop}\label{prop1}
The group $G$ has four orbits on points of $\cH(r,q^2)$:
\begin{itemize}
\item[1)] $\cQ(r,q)$,
\item[2)] $\cO$ consisting of points (not in $\Sigma$) on the extended lines of $\cQ(r,q)$,
\item[3)] $\cE$ consisting of points of $\cH(r,q^2)$ lying on the extended lines of $\Sigma$ that are external to $\cQ(r,q^2)$,
\item[4)] $\cS$ consisting of points of $\cH(r,q^2) \setminus \cQ(r,q)$ lying on the extended lines of $\Sigma$ that are secant to $\cQ(r,q^2)$.
\end{itemize}
In particular $|\cS| = |\cE| = q^{r-1}(q^{r}-1)/2$.
\end{prop} 
\begin{proof}
The group $G$ is transitive on points of $\cQ(r,q)$. Let $P \in \cH(r,q^2) \setminus \cQ(r,q)$, then $P$ lies on a unique extended line $\ell_P$ of $\Sigma$, where $\ell_P$ meets $\cQ(r,q)$ in either $q+1$ or $0$ or $2$ points. Note that $\ell_P \cap \Sigma$ cannot be tangent to $\cQ(r,q)$. Indeed a line of $\Sigma$ that is tangent to $\cQ(r,q)$ at the point $R$, when extended over $\GF(q^2)$, is also a tangent line to $\cH(r,q^2)$ at the point $R$. On the other hand the stabilizer of $\ell_P$ in $G$ permutes in a single orbit the $q^2-q$ or $q+1$ or $q-1$ points of $\cH(r,q^2) \setminus \cQ(r,q)$ on $\ell_P$, respectively, see also \cite[Proposition 2.2]{CP1}. The result now follows from the fact that the group $G$ is transitive on lines of $\Sigma$ that are either contained in $\cQ(r,q)$ or external to $\cQ(r,q)$ or secant to $\cQ(r,q)$.
\end{proof}

\begin{cor}
$\cH(r,q^2) \cap \Sigma = \cQ(r,q)$ and $\cH(r,q^2) \cap \cQ(r,q^2) = \cQ(r,q) \cup \cO$.
\end{cor} 

\begin{lemma}\label{lemma2}
Let $\ell$ be line of $\cH(r,q^2)$, then the following possibilities occur:
\begin{itemize}
\item[i)] $\ell$ has $q+1$ points in common with $\cQ(r,q)$ and $q^2-q$ points in common with $\cO$,
\item[ii)] $\ell$ has a point in common with $\cQ(r,q)$ and $q^2$ points in common with $\cE$,
\item[iii)] $\ell$ has a point in common with $\cQ(r,q)$ and $q^2$ points in common with $\cS$,
\item[iv)] $\ell$ has a point in common with $\cQ(r,q)$, $r \ge 6$, and $q^2$ points in common with $\cO$, 
\item[v)] $\ell$ is contained in $\cO$, here $r \ge 8$,
\item[vi)] $\ell$ has a point in common with $\cO$ and $q^2$ points in common with $\cE$, here $r \ge 6$,
\item[vii)] $\ell$ has a point in common with $\cO$ and $q^2$ points in common with $\cS$, here $r \ge 6$,
\item[viii)] $\ell$ has $2$ points in common with $\cO$ and $(q^2-1)/2$ points in common with both $\cE$ and $\cS$,
\item[ix)] $\ell$ has $(q^2+1)/2$ points in common with both $\cE$ and $\cS$.
\end{itemize} 
\end{lemma}
\begin{proof}
A line of $\PG(r,q^2)$ meets $\Sigma$ in $0$, $1$ or $q+1$ points. Let $\ell$ be a line of $\cH(r,q^2)$. If $|\ell \cap \Sigma| = q+1$, then $\ell \cap \Sigma$ is a line of $\cQ(r,q)$ and hence $|\ell \cap \cO| = q^2-q$. If $|\ell \cap \Sigma| \ne q+1$, then either $|\ell \cap \cQ(r,q)| = 1$ or $|\ell \cap \cQ(r,q)| = 0$. 

If the former case occurs, let $P = \cQ(r,q) \cap \ell$. Then $\langle \ell, \tau(\ell) \rangle$ is a plane $\pi$ of $\PG(r,q^2)$ meeting $\Sigma$ in a Baer subplane $\pi_0$. Note that $P \in \pi_0$ and $\pi \subset P^\perp = P^{\perp_h}$. It follows that $\pi_0$ share with $\cQ(r,q)$ either the point $P$, and $\ell$ is a line of type $ii)$, or two lines of $\pi_0$ through $P$, and $\ell$ is a line of type $iii)$, or $\pi_0$ is contained in $\cQ(r,q)$, $r \ge 6$, and $\ell$ is a line of type $iv)$. Note that the plane $\pi_0$ cannot intersect $\cQ(r,q)$ in a line through $P$, otherwise we would find a line of $\Sigma$ that is tangent to $\cQ(r,q)$ and such that when extended over $\GF(q^2)$ should be a secant line to $\cH(r,q^2)$, a contradiction.    

If the latter case occurs, then $|\tau(\ell) \cap \Sigma| = 0$ and $\langle \ell, \tau(\ell) \rangle$ is a solid $\Pi$ of $\PG(r,q^2)$ meeting $\Sigma$ in a $\PG(3,q)$, say $\Pi_0$. If $\Pi_0 \subset \cQ(r,q)$, $r \ge 8$, and hence $\Pi \subset \cH(r,q^2)$, then $\ell$ is a line of type $v)$. Assume that $\Pi_0 \not\subset \cQ(r,q)$. Then the possibilities for $\Pi_0 \cap \cQ(r,q)$ are listed in \cite[Table 15.4]{H1} and, from \cite[Table 19.1]{H1}, we can say that $\Pi \cap \cH(r,q^2)$ is either a non--degenerate Hermitian surface $\cH(3,q^2)$ or $r \ge 6$ and it consists of $q+1$ planes through a line. Note that, if $\Pi \cap \cH(r,q^2)$ consists of $q+1$ planes through a line $s$, than the line $s$ has to be an extended line of $\cQ(r,q)$. Indeed $s$ has to be fixed by $\tau$. Moreover, every point of $\cH(r,q^2) \cap (\Pi \setminus \Pi_0)$ lies on a unique extended line of $\Pi_0$ and such a line shares $0$, $2$ or $q+1$ points with $\Pi_0 \cap \cQ(r,q)$. Taking into account the fact that the extended lines of $\Pi_0$ meeting both $\ell$ and $\tau(\ell)$ give rise to a line--spread of $\Pi_0$, we have that necessarily one of the possibilities described below occurs. The set $\Pi_0 \cap \cQ(r,q)$ consists of the line $s$ and $\Pi \cap \cH(r,q^2)$ consists of $q+1$ planes through the line $s$. In this case $\ell$ is a line of type $vi)$. The set $\Pi_0 \cap \cQ(r,q)$, $r \ge 6$, consists of two planes through the line $s$ and $\Pi \cap \cH(r,q^2)$ consists of $q+1$ planes through the line $s$. In this case $\ell$ is a line of type $vii)$. The set $\Pi_0 \cap \cQ(r,q)$ is a hyperbolic quadric $\cQ^+(3,q)$ and $\Pi \cap \cH(r,q^2)$ is a non--degenerate Hermitian surface $\cH(3,q^2)$. In this case $\ell$ is a line of type $viii)$, see also \cite[Proposition 2.2, Lemma 3.1]{C}. Finally, the set $\Pi_0 \cap \cQ(r,q)$ is an elliptic quadric $\cQ^-(3,q)$ and $\Pi \cap \cH(r,q^2)$ is a non--degenerate Hermitian surface $\cH(3,q^2)$. In this case $\ell$ is a line of type $ix)$. 
\end{proof}

\begin{lemma}\label{lemma3}
1) Through a point of $\cQ(r,q)$ there pass $\theta_{r-3,q}$ lines of type $i)$, $q(q^{r-2}-1)(q^{r-4}-1)/(q^2-1)$ lines of type $iv)$, $q^{r-3}(q^{r-2}-1)/2$ lines of type $ii)$ and $q^{r-3}(q^{r-2}-1)/2$ lines of type $iii)$.

2) Through a point of $\cO$ there pass one line of type $i)$, $q^2 (q^{r-4}-1)/(q-1)$ lines of type $iv)$, $q^3(q^{r-6}-1)(q^{r-4}-1)/(q^2-1)$ lines of type $v)$, $q^{r-3}(q^{r-4}-1)/2$ lines of type $vi)$, $q^{r-3}(q^{r-4}-1)/2$ lines of type $vii)$ and $q^{2r-5}$ lines of type $viii)$.   

3) Through a point of $\cE$ there pass $\theta_{r-3,q}$ lines of type $ii)$, $q(q^{r-2}-1)(q^{r-4}-1)/(q^2-1)$ lines of type $vi)$, $q^{r-3}(q^{r-2}-1)/2$ lines of type $viii)$ and $q^{r-3}(q^{r-2}-1)/2$ lines of type $ix)$. 

4) Through a point of $\cS$ there pass $\theta_{r-3,q}$ lines of type $iii)$, $q(q^{r-2}-1)(q^{r-4}-1)/(q^2-1)$ lines of type $vii)$, $q^{r-3}(q^{r-2}-1)/2$ lines of type $viii)$ and $q^{r-3}(q^{r-2}-1)/2$ lines of type $ix)$. 
\end{lemma}
\begin{proof}
Let $P \in \cH(r,q^2)$ and, if $P \notin \cQ(r,q)$, let $\ell_{P}$ be the unique extended line of $\Sigma$ through $P$. We distinguish several cases.

\medskip
\fbox{$P \in \cQ(r,q)$}
\medskip

\noindent
Let $P$ be a point of $\cQ(r,q)$ and let $\Gamma$ be an $(r-2)$--space of $\PG(r,q^2)$ contained in $P^{\perp_h} = P^\perp$, not containing $P$ and such that $\Gamma \cap \Sigma$ is an $(r-2)$--space of $\Sigma$. Then $\Gamma \cap \cQ(r,q)$ is a parabolic polar space $\cQ(r-2,q)$ embedded in $\cQ(r,q)$ and $P^\perp \cap \cQ(r,q)$ is a cone having as vertex the point $P$ and as basis $\cQ(r-2,q)$. On the other hand $\Gamma \cap \cH(r,q^2)$ is a Hermitian polar space $\cH(r-2,q^2)$ embedded in $\cH(r,q^2)$ and $P^{\perp_h} \cap \cH(r,q^2)$ is a cone having as vertex the point $P$ and as basis $\cH(r-2,q^2)$, where $\cQ(r-2,q) \subset \cH(r-2,q^2)$. Let $G'$ be the group of projectivities of $\Gamma$ isomorphic to $\PGO(r-1,q)$ stabilizing the parabolic polar space $\cQ(r-2,q)$. From Proposition \ref{prop1}, the group $G'$ has four orbits on points of $\cH(r-2,q^2)$, say $\cQ(r-2,q)$, $\cO'$, $\cE'$ and $\cS'$, where $\cQ(r-2,q) \subset \cQ(r,q)$, $\cO' \subset \cO$, $\cE' \subset \cE$, $\cS' \subset \cS$. Let $R$ be a point of $\cH(r-2,q^2)$ and let $g$ be the line of $\cH(r,q^2)$ joining $P$ and $R$. If $R \in \cQ(r-2,q)$, then $g$ is a line of type $i)$. If $R \in \cO'$, then the line $\ell_{R}$ is an extended line of $\cQ(r-2,q)$. It follows that $\langle \ell_{R}, g \rangle$ is a plane $\sigma$ of $\cH(r,q^2)$, $r \ge 6$, such that $\sigma \cap \Sigma$ is a Baer subplane and hence $g$ is a line of type $iv)$. If $R \in \cE'$, then $\ell_{R}$ is an extended line of $\Gamma \cap \Sigma$ such that $\ell_{R} \cap \Sigma$ is external to $\cQ(r-2,q)$. Since the plane $\langle P, \ell_{R} \rangle$ meets $\Sigma$ in a Baer subplane which in turn shares with $\cQ(r,q)$ solely the point $P$, it follows that $g$ is a line of type $ii)$. Analogougly, if $R \in \cS'$, we have that $g$ is a line of type $iii)$. 

\medskip
\fbox{$P \in \cO$}
\medskip

\noindent
If $P \in \cO$, then let $\Gamma'$ be an $(r-4)$--space of $\PG(r,q^2)$ contained in $\ell_P^{\perp_h} = \ell_P^\perp$, not containing $\ell_P$ and such that $\Gamma' \cap \Sigma$ is an $(r-4)$--space of $\Sigma$. Then $\Gamma' \cap \cQ(r,q)$ is a parabolic polar space $\cQ(r-4,q)$ embedded in $\cQ(r,q)$ and $\ell_P^\perp \cap \cQ(r,q)$ is a cone having as vertex the line $\ell_P$ and as basis $\cQ(r-4,q)$. On the other hand $\Gamma' \cap \cH(r,q^2)$ is a Hermitian polar space $\cH(r-4,q^2)$ embedded in $\cH(r,q^2)$ and $\ell_P^{\perp_h} \cap \cH(r,q^2)$ is a cone having as vertex the line $\ell_P$ and as basis $\cH(r-4,q^2)$, where $\cQ(r-4,q) \subset \cH(r-4,q^2)$. Note that if a line of $\cH(r,q^2)$ through $P$ intersect $\cQ(r,q)$ in one point, then such a point lies on $(\ell_P^{\perp} \cap \Sigma) \setminus \ell_P$. Let $G''$ be the group of projectivities of $\Gamma'$ isomorphic to $\PGO(r-3,q)$ stabilizing the parabolic polar space $\cQ(r-4,q)$. From Proposition \ref{prop1}, the group $G''$ has four orbits on points of $\cH(r-4,q^2)$, say $\cQ(r-4,q)$, $\cO''$, $\cE''$ and $\cS''$, where $\cQ(r-4,q) \subset \cQ(r,q)$, $\cO'' \subset \cO$, $\cE'' \subset \cE$, $\cS'' \subset \cS$. 

The hyperplane $P^{\perp_h}$ meets $\cH(r,q^2)$ in a cone having as vertex the point $P$ and as basis $\cH(r-2,q^2)$. In particular, $\cH(r-2,q^2)$ can be chosen in such a way that $\cQ(r-4,q) \subseteq \cH(r-4,q^2) \subseteq \cH(r-2,q^2)$ and that the point $T:=\cH(r-2,q^2) \cap \ell_P$ belongs to $\Sigma$. Let $\Lambda$ be the $(r-2)$--space of $\PG(r,q^2)$ containing $\cH(r-2,q^2)$, then $\Lambda \cap \Sigma$ is the $(r-3)$--space of $\Sigma$ containing $T$ and $\cQ(r-4,q)$. Hence such an $(r-3)$--space, when extended over $\GF(q^2)$ coincides with $\langle T, \Gamma' \rangle$, which meets $\cH(r-2,q^2)$ in the cone having as vertex the point $T$ and as basis $\cH(r-4,q^2)$. It follows that $\cH(r-2,q^2) \setminus \langle T, \Gamma' \rangle$ consists of $q^{2r-5}$ points.

Let $R$ be a point of $\cH(r-2,q^2)$ and let $g$ be the line of $\cH(r,q^2)$ joining $P$ and $R$. If $R = T$, then $g = \ell_P$ and $\ell_P$ is the unique line of type $i)$ passing through $P$. Assume that $R \ne T$. If $R$ lies on a line joining $T$ with a point of $\cQ(r-4,q)$, then $g$ is a line of type $iv)$. If $R$ lies on line joining $T$ with a point of $\cO''$, then the line $\ell_{R}$ is an extended line of a parabolic polar space $\cQ(r-4,q)$, $r \ge 8$, embedded in $\ell_P^\perp \cap \Sigma$ and disjoint from $\ell_P$. It follows that $\langle \ell_P, \ell_{R} \rangle$ is a solid $\sigma'$ of $\cH(r,q^2)$, $r \ge 8$, such that $\sigma' \cap \Sigma$ is a Baer subgeometry isomorphic to $\PG(3,q)$ and hence $g$ is a line of type $v)$. If $R$ lies on a line joining $T$ with a point of $\cE''$, then $\ell_{R}$ is an extended line of $\ell_P^\perp \cap \Sigma$ such that $\ell_{R} \cap \Sigma$ is external to $\cQ(r,q)$. Since the plane $\langle P, \ell_{R} \rangle$ meets $\Sigma$ in a Baer subplane which in turn shares with $\cQ(r,q)$ solely the point $P$, it follows that $g$ is a line of type $vi)$. Analogougly, if $R$ lies on a line joining $T$ with a point of $\cS''$, we have that $g$ is a line of type $vii)$. Finally, let $R$ be a point of $\cH(r-2,q^2) \setminus \langle T, \Gamma' \rangle$. Then the line joining $R$ and $\tau(R)$ is disjoint from $\langle T, \Gamma' \rangle$ and hence the solid generated by $\ell_P$ and the line $\ell_R$ meets $\cQ(r,q)$ in a hyperbolic quadric $\cQ^+(3,q)$. Hence $g$ is a line of type $viii)$.

\medskip
\fbox{$P \in \cE$}
\medskip

\noindent
If $P \in \cE$, then $|\ell_{P} \cap \cQ(r,q)| = 0$, $\ell_{P}^\perp \cap \cQ(r,q) = \cQ(r-2,q)$, $\ell_{P}^{\perp_h} \cap \cH(r,q^2) = \cH(r-2,q^2)$, where $\cQ(r-2,q) \subset \cH(r-2,q^2)$. Moreover ${P}^{\perp_h}$ meets $\cH(r,q^2)$ in a cone having as vertex the point $P$ and as basis $\cH(r-2,q^2)$. Let $G'$ be the group of projectivities of $\ell_{P}^\perp \cap \Sigma$ isomorphic to $\PGO(r-1,q)$ stabilizing the parabolic polar space $\cQ(r-2,q)$. From Proposition \ref{prop1}, the group $G'$ has four orbits on points of $\cH(r-2,q^2)$, say $\cQ(r-2,q)$, $\cO'$, $\cE'$ and $\cS'$, where $\cQ(r-2,q) \subset \cQ(r,q)$, $\cO' \subset \cO$, $\cE' \subset \cE$, $\cS' \subset \cS$. Let $R$ be a point of $\cH(r-2,q^2)$ and let $g$ be the line of $\cH(r,q^2)$ joining $P$ and $R$. If $R \in \cQ(r-2,q)$, then $g$ is a line of type $ii)$. If $R \in \cO'$, then the line $\ell_{R}$ is an extended line of $\cQ(r-2,q)$. It follows that $\langle \ell_{R}, g \rangle$ is a plane $\sigma$ of $\cH(r,q^2)$ such that $\sigma \cap \Sigma = \ell_{R}$ and hence $g$ is a line of type $vi)$. If $R \in \cE'$, then $\ell_{R}$ is an extended line of $\Sigma$ such that $\ell_{R} \cap \Sigma$ is external to $\cQ(r,q)$. Let $\Lambda$ be the solid generated by $\ell_{P}$ and $\ell_{R}$. Since $\ell_{P}^\perp \cap \Lambda = \ell_{R}$ and $\ell_{R}^\perp \cap \Lambda = \ell_{P}$, we have that $\Lambda \cap \cQ(r,q) = \cQ^+(3,q)$ and $g$ is a line of type $viii)$. Analogougly, if $R \in \cS'$, we have that $\Lambda \cap \cQ(r,q) = \cQ^-(3,q)$ and $g$ is a line of type $ix)$. A similar argument holds if $P \in \cS$. 
\end{proof}

\begin{theorem}
Each of the four $G$--orbits $\cQ(r,q)$, $\cO$, $\cE$, $\cS$ is an $i$--tight set of $\cH(r,q^2)$ with parameter $i= q+1$, $q^{r-1}-q$, $(q^{r+1}-q^{r-1})/2$, $(q^{r+1}-q^{r-1})/2$, respectively.
\end{theorem}
\begin{proof}
Let $P$ be a point of $\cH(r,q^2)$. Taking into account Lemma \ref{lemma2} and Lemma \ref{lemma3}, we have that
$$
|P^{\perp_h} \cap \cQ(r,q)| =
\begin{cases}
q |\cQ(r-2,q)| + 1 = (q+1) \theta_{\frac{r-4}{2},q^2} + q^{r-2} & \mbox{ if } P \in \cQ(r,q), \\

q+1 + q^2 \frac{q^{r-4}-1}{q-1} = \frac{q^{r-2}-1}{q-1} = (q+1) \theta_{\frac{r-4}{2},q^2} & \mbox{ if } P \in \cO, \\

|\cQ(r-2,q)| = (q+1) \theta_{\frac{r-4}{2},q^2} & \mbox{ if } P \in \cE \cup \cS, \\
\end{cases} 
$$
$$
|P^{\perp_h} \cap \cO| =
\begin{cases}
q^3\frac{(q^{r-2}-1)(q^{r-4}-1)}{q^2-1} + (q^2-q) \theta_{r-3,q} = (q^{r-1}-q) \theta_{\frac{r-4}{2},q^2} & \mbox{ if } P \in \cQ(r,q), \\

(q^2-q) + (q^2-1) q^2 \theta_{r-3,q} + q^5 (q^{r-6}-1)\theta_{\frac{r-6}{2},q^2} +  \\
+  q^{2r-5} = (q^{r-1}-q) \theta_{\frac{r-4}{2},q^2} + q^{r-2} & \mbox{ if } P \in \cO, \\

q \frac{(q^{r-2}-1)(q^{r-4}-1)}{q^2-1} + 2 q^{r-3} \frac{q^{r-2}-1}{2} = (q^{r-1}-q) \theta_{\frac{r-4}{2},q^2} & \mbox{ if } P \in \cE \cup \cS, \\
\end{cases} 
$$
$$ 
|P^{\perp_h} \cap \cE| =
\begin{cases}
q^{r-1} \frac{q^{r-2}-1}{2} = \frac{q^{r+1}-q^{r-1}}{2} \theta_{\frac{r-4}{2},q^2} & \mbox{ if } P \in \cQ(r,q), \\

q^{r-1} \frac{q^{r-4}-1}{2} + \frac{q^2-1}{2} q^{2r-5} =  \frac{q^{r+1}-q^{r-1}}{2} \theta_{\frac{r-4}{2},q^2} & \mbox{ if } P \in \cO, \\

(q^2-1) \frac{q^{r-2}-1}{q-1} + (q^2-1) q \frac{(q^{r-2}-1)(q^{r-4}-1)}{q^2-1} + \\
+ q^{r-3} \frac{q^2-3}{2} \frac{q^{r-2}-1}{2} + q^{r-3} \frac{q^2-1}{2} \frac{q^{r-2}-1}{2} + 1 = \frac{q^{r+1}-q^{r-1}}{2} \theta_{\frac{r-4}{2},q^2} + q^{r-2} & \mbox{ if } P \in \cE, \\

q^{r-3} \frac{q^2-1}{2} \frac{q^{r-2}-1}{2} + q^{r-3} \frac{q^2+1}{2} \frac{q^{r-2}-1}{2} = \frac{q^{r+1}-q^{r-1}}{2} \theta_{\frac{r-4}{2},q^2} & \mbox{ if } P \in \cS, \\
\end{cases} 
$$
$$
|P^{\perp_h} \cap \cS| =
\begin{cases}
q^{r-1} \frac{q^{r-2}-1}{2} = \frac{q^{r+1}-q^{r-1}}{2} \theta_{\frac{r-4}{2},q^2} & \mbox{ if } P \in \cQ(r,q), \\

q^{r-1} \frac{q^{r-4}-1}{2} + \frac{q^2-1}{2} q^{2r-5} =  \frac{q^{r+1}-q^{r-1}}{2} \theta_{\frac{r-4}{2},q^2} & \mbox{ if } P \in \cO, \\

q^{r-3} \frac{q^2-1}{2} \frac{q^{r-2}-1}{2} + q^{r-3} \frac{q^2+1}{2} \frac{q^{r-2}-1}{2} = \frac{q^{r+1}-q^{r-1}}{2} \theta_{\frac{r-4}{2},q^2} & \mbox{ if } P \in \cE, \\

(q^2-1) \frac{q^{r-2}-1}{q-1} + (q^2-1) q \frac{(q^{r-2}-1)(q^{r-4}-1)}{q^2-1} + \\
+ q^{r-3} \frac{q^2-3}{2} \frac{q^{r-2}-1}{2} + q^{r-3} \frac{q^2-1}{2} \frac{q^{r-2}-1}{2} + 1 = \frac{q^{r+1}-q^{r-1}}{2} \theta_{\frac{r-4}{2},q^2} + q^{r-2} & \mbox{ if } P \in \cS. \\
\end{cases} 
$$
\end{proof}

An {\em $m$--cover of lines of $\cQ(r,q)$} is a set of lines $\cL$ of $\cQ(r,q)$ such that every point of $\cQ(r,q)$ contains $m$ lines of $\cL$. In this section we investigate properties of $m$--covers of $\cQ(r,q)$. Similar arguments to that used in the proof of Lemma \ref{lemma1} yield the following. 

\begin{lemma}\label{lemma4}
Let $\cL$ be an $m$--cover of $\cQ(r,q)$, $r \ge 4$ even, then
\begin{itemize} 
\item[$i)$] $\cL$ contains $m \theta_{r-1,q}/(q+1)$ lines, 
\item[$ii)$] a hyperplane $H$ of $\PG(r,q)$ contains $m \theta_{r-3,q}/(q+1)$ lines of $\cL$, if $H$ is tangent, $m (q^{\frac{r}{2}}-1)(q^{\frac{r-4}{2}}+1)/(q^2-1)$, if $H \cap \cQ(r,q) = \cQ^+(r-1,q)$, $m (q^{\frac{r}{2}}+1)(q^{\frac{r-4}{2}}-1)/(q^2-1)$, if $H \cap \cQ(r,q) = \cQ^-(r-1,q)$,  
\item[$iii)$] let $\Sigma$ be an $(r-2)$--space of $\PG(r,q)$, where $\Sigma = \ell^\perp$ and $|\ell \cap \cQ(r,q)| \in \{0, 2, q+1\}$. If $x$ is the number of lines of $\cL$ contained in $\Sigma$ and $y$ is the number of lines of $\cL$ meeting $\Sigma$ in one point, then $x(q+1)+y = m \theta_{r-3,q}$.  
\end{itemize}
\end{lemma}

We introduce the following definitions.

	\begin{defin}
An $m$--cover $\cL$ of ${\cQ}(r,q)$ is said to be {\em $(r-2)$--dual of parabolic type $I$} if 
$$\vert\{r\in \cL:r\subseteq \ell^\perp\}\vert=
	\begin{cases}
\frac{m \theta_{r-5,q}}{q+1}+q^{r-4} & \mbox{ if } \ell\in{\cL}, \\
\frac{m \theta_{r-5,q}}{q+1} & \mbox{ if } \ell\not\in{\cL} \mbox{ and } |\ell \cap \cQ(r,q)| \in \{0,2,q+1\}.  
	\end{cases} 
$$		
	\end{defin}

	\begin{defin}
An $m$--cover $\cL$ of ${\cQ}(r,q)$ is said to be {\em $(r-2)$--dual of parabolic type $II$} if 
$$\vert\{r\in \cL:r\subseteq \ell^\perp\}\vert=
	\begin{cases}
\frac{m \theta_{r-5,q}}{q+1}+q^{r-4} & \mbox{ if } \ell\in{\cL}, \\
\frac{m \theta_{r-5,q}}{q+1} & \mbox{ if } \ell\not\in{\cL} \mbox{ and } |\ell \cap \cQ(r,q)| \in \{0,q+1\}.  
	\end{cases} 
$$		
	\end{defin}

Basing on the definitions above we have the following results.  
	
\begin{theorem}\label{main4}
Let $\cL$ be an $m$--cover of ${\cQ}(r,q)$ and let ${\bar \cL}$ be the set of points of ${\cH}(r,q^2)$ lying on the extended lines of $\cL$. Then ${\bar \cL}$ is an $(m(q^2-q)+q+1)$--tight set of ${\cH}(r,q^2)$ if and only if $\cL$ is an $(r-2)$--dual cover of parabolic type $I$. 
\end{theorem}
\begin{proof}
Assume that ${\bar\cL}$ is an $(m(q^2-q)+q+1)$--tight set of ${\cal H}(r,q^2)$. By construction, ${\bar\cL}$ is contained in $\cQ(r,q) \cup {\cal O}$. Hence, if $P\in {\bar \cL} \setminus {\cal Q}(r,q)$ then $P^{\perp_h} \cap {\cQ(r,q)} = \ell_P^{\perp}$, where $\ell_P$ is the unique extended line of $\cL$ through $P$. If $P \in \cO$, then
$$
\vert P^{\perp_h}\cap{\bar\cL}\vert=(m(q^2-q)+q+1) \theta_{\frac{r-4}{2},q^2} + q^{r-2} = x (q^2-q) + \theta_{r-3,q} + \vert \cL \vert - x - y, 
$$
where $x$ is the the number of extended lines of $\cL$ contained in $P^{\perp_h}$ (hence contained in $\ell_P^\perp$) and $y$ is the number of extended lines of $\cL$ meeting $P^{\perp_h}$ in a point of $\cQ(r,q)$. Then, taking into account Lemma \ref{lemma4}, $iii)$, we have 
$$
xq^2=m \frac{q^{r-2}-q^2}{q^2-1}+q^{r-2},
$$
and hence $x= \frac{m \theta_{r-5,q}}{q+1}+q^{r-4}$.
If $P\in{\cH}(r,q^2)\setminus{\bar\cL}$, then 
$$
\vert P^{\perp_h}\cap{\bar\cL}\vert = (m(q^2-q)+q+1)\theta_{\frac{r-4}{2},q^2} = x(q^2-q) + \theta_{{r-3},q} + \vert \cL \vert - x - y,
$$ 
where again $x$ is the the number of extended lines of $\cL$ contained in $P^{\perp_h}$ (hence contained in $\ell_P^{\perp}$) and $y$ is the number of extended lines of $\cL$ meeting $P^{\perp_h}$ in a point of $\cQ(r,q)$. In this case $x = \frac{m \theta_{r-5,q}}{q+1}$.

Viceversa, assume that $\cL$ is an $(r-2)$--dual cover of parabolic type I. If $P \in {\cal Q}(r,q)$, then 
$$
\vert P^{\perp}\cap{\bar\cL}\vert=\vert P^{\perp_h}\cap{\bar\cL}\vert=m (q^2-q) \frac{\theta_{r-3,q}}{q+1} + \theta_{r-2,q} = (m(q^2-q)+q+1) \theta_{\frac{r-4}{2},q^2} + q^{r-2}. 
$$
If $P\in{\cH}(r,q^2)\setminus{\cal Q}(r,q)$ then $P^{\perp}={\tau(P)}^{\perp_h}$, where $\langle P,\tau(P)\rangle$ is the unique extended line of $\Sigma$ through $P$. Hence $\tau(P)\in\bar\cL$ if and only if $P\in\bar\cL$. Since
$$
P^{\perp_h}\cap{\bar\cL} = x(q^2-q) + \theta_{r-3,q} + \vert\cL\vert - x - y,
$$ 
where $x$ is the the number of lines of $\cL$ contained in $P^{\perp_h}$ (hence contained in $\ell_P^\perp$) and $y$ is the number of lines of $\cL$ meeting $P^{\perp_h}$ in a point of $\cQ(r,q)$, it follows that $\vert P^{\perp_h}\cap{\bar\cL}\vert$ is either 
$$
(m(q^2-q)+q+1) \theta_{\frac{r-4}{2},q^2} + q^{r-2}
$$ 
or 
$$
(m(q^2-q)+q+1) \theta_{\frac{r-4}{2},q^2},
$$
according as $P$ lies or does not lie on $\bar\cL$. Therefore ${\bar \cL}$ is a $(m(q^2-q)+q+1)$--tight set of ${\cal H}(r,q^2)$.	
\end{proof}

The proof of the following result is similar to that given for Theorem \ref{main4} and hence we omit it.

	\begin{theorem}\label{main5}
Let $\cL$ be an $m$--cover of ${\cal Q}(r,q)$ and let ${\bar \cL}$ be the set of points of ${\cal Q}(r,q^2)$ lying on the extended lines of $\cL$. Then ${\bar \cL}$ is an $(m(q^2-q)+q+1)$--tight set of ${\cal Q}(r,q^2)$ if and only if $\cL$ is an $(r-2)$--dual cover of parabolic type $II$.
	\end{theorem}
	
It is straightforward to remark that an $(r-2)$--dual cover of parabolic type $I$ is an $(r-2)$--dual cover of parabolic type $II$ and if $\cL$ consists of all the lines of $\cQ(r,q)$, then $\cL$ is an $(r-2)$--dual $\theta_{r-3,q}$--cover of parabolic type $I$. A more interesting example is the following. Let $\PGO(7,q)$ be the projective orthogonal group of $\cQ(6,q)$. The Cartan Dickson Chevalley exceptional group $G_2(q)$ is a subgroup of $\PGO(7,q)$. It occurs as the stabilizer in $\PGO(7,q)$ of a configuration of points, lines and planes of $\cQ(6,q)$. The group $G_2(q)$ is also contained in the automorphism group of a classical generalized hexagon, called the {\em split Cayley hexagon} and denoted by $H(q)$, see \cite{De}. The points of $H(q)$ are all the points of $\cQ(6,q)$ and the lines of $H(q)$ are certain lines of $\cQ(6,q)$. The number of points of $\cQ(6,q)$ is $(q^6-1)/(q-1)$, and this is also the number of lines and planes of $\cQ(6,q)$ involved in $H(q)$. The generalized hexagon $H(q)$ has the following properties \cite[p. 33]{De}: 
\begin{itemize}
\item through any point of $H(q)$ there pass $q+1$ lines of $H(q)$, and any line of $H(q)$ contains $q+1$ points of $H(q)$; 
\item the $q+1$ lines of $H(q)$ through a point of $H(q)$ are contained in a unique plane of $H(q)$;
\item a plane of $\cQ(6,q)$ either contains $q+1$ lines of $H(q)$ and in this case it is a plane of $H(q)$ or it contains no line of $H(q)$;
\item through a line of $H(q)$, there pass $q+1$ planes of $H(q)$, whereas through a line of $\cQ(6,q)$ that is not a line of $H(q)$, there is exactly one plane of $H(q)$;
\item an elliptic quadric $\cQ^-(5,q)$ embedded in $\cQ(6,q)$ contains exactly $q^3+1$ pairwise disjoint lines of $H(q)$. 
\end{itemize}
 
\begin{prop}\label{splitodd}
The set of lines of $H(q)$ is a $3$--dual cover of parabolic type $I$.	
\end{prop}
\begin{proof}
Let $\ell$ be a line of $\cQ(6,q)$. Then $\ell^\perp$ contains the $q+1$ planes of $\cQ(6,q)$ through $\ell$. If $\ell$ is a line of $H(q)$, then every plane through $\ell$ contains $q$ lines of $H(q)$ distinct from $\ell$. If $\ell$ is a line of $\cQ(6,q)$ not of $H(q)$, then through $\ell$ there is a unique plane of $H(q)$. Hence in the former case $\ell^\perp$ contains $q^2+q+1$ lines of $H(q)$, while in the latter case $\ell^\perp$ contains $q+1$ lines of $H(q)$. Let $\ell$ be a line of $\PG(6,q)$ such that $|\ell \cap \cQ(6,q)| \in \{0,2\}$. Then $\ell^\perp$ meets $\cQ(6,q)$ in a parabolic quadric $\cQ(4,q)$. Let $\cQ^-(5,q)$ be an elliptic quadric contained in $\cQ(6,q)$ and containing $\cQ(4,q)$. Since the $q^3+1$ lines of $H(q)$ contained in $\cQ^-(5,q)$ are pairwise disjoint, they give rise to a line--spread of $\cQ^-(5,q)$ and hence to an ovoid of the dual of $\cQ^-(5,q)$. On the other hand, the set of lines of $\cQ(4,q)$ corresponds to a $(q+1)$--tight set of the dual of $\cQ^-(5,q)$. It follows that there are exactly $q+1$ lines of $H(q)$ contained in $\ell^\perp$.       
\end{proof}

	\begin{remark}
The fact that $\cQ(r,q)$ is a $(q+1)$--tight set of $\cH(r,q^2)$ has been already proved in \cite[Theorem 4.1]{DM}.  In the same paper, the authors conjecture that every $(q+ 1)$--tight  set  of $\cH(4,q^2)$ that  is  not  the  union  of $q+ 1$ lines is  the  set  of  points  of  an  embedded $\cW(3,q)$,  or,  when $q$ is  odd,  the  set  of  points  of  an embedded $\cW(3,q)$ or $\cQ(4,q)$.  With the aid of MAGMA \cite{BCP} we were able to construct counterexamples to this conjecture when $q=2,3$.
Assume that $\cH(4,q^2)$ has equation $X_1^{q+1}+X_2^{q+1}+X_3^{q+1}+X_4^{q+1}+X_5^{q+1}=0$.
The point set $S:=\{U_1,\dots,U_5\}$, where $U_i$ is the point of $\PG(4,q^2)$ having $1$ at the $i$--th position and $0$ elsewhere forms a self--polar simplex with respect to $\cH(4,q^2)$. Then, the $10$ lines joining two points of $S$ are secant lines of $\cH(4,q^2)$. Let $G={(C_{q+1})}^4 \rtimes Sym_5$ be the stabilizer of $S$ in $\PGU(5,q^2)$.

\medskip
\fbox{$q=2$}
\medskip

In this case $\vert G\vert = 4860$ and $G$ has a unique subgroup, say $H$ of index $6$. The group $H$ acts transitively on $S$ and has two orbits ${\cL}_i$, $i=1,2$, of size $5$ on the $10$ secant lines joining two points of $S$: 
$$
\cL_1:=\{U_1U_4,U_1U_5,U_3U_5,U_2U_4,U_2U_3\}\\
\quad \cL_2:=\{U_4U_5,U_3U_4,U_1U_3,U_2U_5,U_1U_2\}.
$$
Each line of $\cL_i$ meets $\cH(4,4)$ in $3$ points giving rise to a subset of $15$ points that is a transitive $3$--tight set that is neither union of generators nor a subquadrangle of order $2$. 

\medskip
\fbox{$q=3$}
\medskip

In this case each of the $10$ lines joining two points of $S$ meets $\cH(4,9)$ in $4$ points giving rise to a subset of $40$ points forming a transitive $4$--tight set that is neither union of generators nor a subquadrangle of order $3$.

Similarly, using the stabilizer of a self--polar simplex in $\PGU(7,4)$ we found a $3$--tight set of $\cH(6,4)$ consisting of the $63$ points on the $21$ secant lines joining two points of the simplex.

	\end{remark}


\end{document}